\documentclass[reqno, 10pt]{article}
\usepackage[utf8]{inputenc}
\usepackage{mathtools,mathrsfs,amssymb,amsthm,amsmath,
	bm}
\usepackage{hyperref} 
\usepackage{url}
\usepackage[a4paper, left = 1.25in, right = 1.25in, top = 1.25in, bottom = 1.25in]{geometry} 
\usepackage{comment}
\usepackage{physics}\usepackage{dsfont}
\usepackage[british]{babel}

\newtheorem{thm}{Theorem}
\newtheorem{prop}{Proposition}
\theoremstyle{remark}
\newtheorem*{rem}{Remark}

\def\[#1\]{\begin{align*}#1\end{align*}}

\newcommand{\R}{\mathbb{R}}
\newcommand{\N}{\mathbb{N}}
\def\C{\mathbb{C}}
\newcommand{\Q}{\mathbb{Q}}
\renewcommand{\L}{\mathbb{L}}
\newcommand{\M}{\mathbb{M}}
\newcommand{\ceq}{\coloneqq}

\newcommand{\eps}{\varepsilon}
\newcommand{\lrtx}[1]{\ \text{#1} \ }
\newcommand{\ltx}[1]{\ \text{#1}}

\newcommand{\I}{\mathds{1}}

\newcommand{\f}{\mathscr{F}}

\title{Generalization and New Proof for Almost Everywhere Convergence to Imply Local Convergence in Measure}
\author{Yu-Lin Chou\thanks{
\fontsize{8}{12pt}\selectfont
Author for correspondence:
Yu-Lin Chou,
Institute of Statistics,
National Tsing Hua University, Hsinchu 30013, Taiwan;
Email: 
\protect\url{y.l.chou@gapp.nthu.edu.tw}.
The author would like to gratefully acknowledge the constructive  comments obtained for the previous drafts of the present paper.
}}
\date{}

\begin{document}
\setlength{\baselineskip}{15pt} 
\maketitle

\begin{abstract}
With a new proof approach we prove in a more general setting the classical convergence theorem that almost everywhere convergence of measurable functions on a finite measure space implies convergence in measure.
Specifically,
we generalize the theorem  for the case where the codomain is a separable metric space and for the case where the limiting map is constant and the codomain is an arbitrary topological space.\\


\end{abstract}

\section{Introduction}
It is a classical result that,
if $f, f_{1}, f_{2}, \dots$ are measurable $\C$-valued functions 
on a finite measure space and
if $f_{n} \to f$ almost everywhere, 
then 
$f_{n} \to f$ in measure.
The importance of the convergence theorem is fully aware.
It would be useful (and also intellectually amusing)
to prove the convergence theorem when the codomain of the maps $f, f_{1}, f_{2}, \dots$
is a metric space or even an arbitrary topological space.
This task is not trivial;
for example,
the usual proof approach, 
for $f$ constant, 
cannot deal with the case when
$\C$ is replaced with an arbitrary topological space.

We give a new proof for the convergence theorem that, to a certain extent, allows of the aforementioned generalization.
At the same time,
although an application of our generalization,
for  purposes such as a probabilistic or statistical one,
is in a sense immediate for 
``well-behaved'' maps as a probability measure is a suitably scaled finite measure,
we provide a counterexample showing  that the result does not necessarily hold if the measurability  of the involved maps is undecided;
difficulty in proving measurability is not unusual in applications such as in the context of asymptotic statistical inference, 
\textit{e.g.} establishing the measurability of a nonlinear least squares estimator in $\R^{k}$ for some integer $k \geq 1$ (Lemma 2 in Jennrich \cite{j}).

\section{Preliminaries}
Throughout, 
we fix a finite measure space $(\Omega, \f, \M)$.

Following the convention of probability theory,
we will in general write for simplicity a set of the form 
$\{ \omega \mid g(\omega) \ltx{has a given property} \}$ as
$\{g \ltx{has the property} \}$.
When written in juxtaposition with a set function,
in particular a measure or an outer measure,
the set 
$\{ g \ltx{has the property} \}$ 
will simply take the form 
$( g $
$\ltx{has the property})$.

If $S$ is a topological space with 
$\mathscr{B}_{S}$ 
the Borel sigma-algebra, 
if $f_{n}: \Omega \to S$ is $(\f, \mathscr{B}_{S})$-measurable for all $n \in \N$,
and if $f: \Omega \to S$ is constant,
then,
regarding the convergence of the sequence $(f_{n})$ to $f$, the involved notion of closeness to $f$ is understood in terms of the open subsets of $S$ that contain (the point of $S$ identified with) $f$. For example,
the definition of convergence in $\M$-measure is to be paraphrased in this case as ``for every open $G \subset S$
containing the constant identified with $f$,
we have $\M (\{ \omega \in \Omega \mid f_{n}(\omega) \notin G \}) = \M ( f_{n} \notin G ) \to 0$.''.

\section{Results}
Given a sequence of subsets of $\Omega$,
we can partition the space $\Omega$ into the limit inferior of the sequence and the limit superior of the sequence of the complements of the subsets of $\Omega$;
this observation is the fundamental proof idea.

\begin{thm}\label{main}
 Let $S$ be a topological space;
 let $f, f_{n}: \Omega \to S$ be measurable-$(\f, \mathscr{B}_{S})$ for all $n \in \N$; let $f_{n} \to f$ almost everywhere with respect to $\M$.
 i) If $S$ is in particular a separable metric space, then $f_{n} \to f$ in $\M$-measure; ii) if $f$ is in  particular a constant map,
 then $f_{n} \to f$ in $\M$-measure.
\end{thm}

\begin{proof}
Let 
$\eps > 0$.

For i),
let $d$ be the separable metric on $S \times S$.
Since $d$ is continuous with respect to the product  $d$-topology,
and since the countable base property of $S$ ensures that
the map $(f_{n},f)$ is measurable with respect to the Borel sigma-algebra generated by the product $d$-topology for all $n \in \N$,
the function $d(f_{n},f)$ is measurable for all $n \in \N$.

Let
$N \in \N$
whenever 
$\liminf_{n \to \infty}\{
d(f_{n}, f) \leq \eps \}$
is empty;
otherwise,
let
$N \ceq \inf_{\omega \in \Omega} \inf J$,
where,
for every $\omega \in \Omega$,
the inner infimum extends over all $J \in \N$ such that $d(f_{j}(\omega), f(\omega)) \leq \eps$
for all $j \geq J$.
Then, 
for all $n \geq N$ we have
\[
  0 
  &\leq 
  \M (d(f_{n}, f) > \eps)\\
  &=
  \M \big(d(f_{n}, f) > \eps, 
 d(f_{m},f) \leq \eps \lrtx{for sufficiently large} m \big) + 
  \M \big(d(f_{n}, f) > \eps,
  \limsup_{m \to \infty}d(f_{m}, f) > \eps \big)\\
  &\leq
  0 + \M \big( \limsup_{m \to \infty}d(f_{m}, f) > \eps \big)\\
  &= \M(\Omega) - \M \big(d(f_{m}, f) \leq \eps \lrtx{for sufficiently  large} m \big)\\
  &\leq \M(\Omega) - \M(f_{n} \to f \ltx{pointwise})\\
  &= 0;
\]
the last equality follows from the convergence assumption.

For ii), 
we rewrite the partition of the finite measure space $\Omega$ used above as 
\[
\{ 
\liminf_{n \to \infty}\{ f_{n} \in G \}, 
\limsup_{n \to \infty}
\{ f_{n} \notin G \}
\}
\]
with $G \ni f$ 
being a given open subset of $S$. 
Then ii) follows from the main argument above with the apparent slight modification. 
\end{proof}

\begin{rem}
For the breadth of some application possibilities of Theorem 1,
we recall that many familiar function spaces can be made a separable metric space,
\textit{e.g.}
the space $\R^{\infty}$ 
(equipped with a usual product metric), 
the spaces $L_{p}(\R^{n})$ (equipped with the metric induced by the usual $L_{p}$-norm) 
for $1 \leq p < +\infty$ and $n \in \N$, 
the space $C$ of all the $\R$-valued continuous functions on $[0,1]$ (equipped with the uniform metric),
and the space $D$ of all the $\R$-valued c{\'a}dl{\'a}g\footnote{By a c{\'a}dl{\'a}g function we mean a function that has left limit and is right continuous everywhere.} 
functions on $[0,1]$ (equipped with the Skorokhod metric, which is a metric derived from the uniform metric),
are separable metric spaces.\footnote{For the separability of $L_{p}(\R^{n})$,
there is a proof given in Brezis \cite{br};
for the separability of each of the other cases, 
there is a proof contained in Billingsley \cite{b}.}
\qed
\end{rem}

If the involved maps $f,f_{n}$ are not all measurable,
or if the measurability is not obvious,
then one may try to circumvent the measurability issue via the outer measure obtained by taking for every $A \subset \Omega$ the infimum of the set $\{ \M(B) \mid B \supset A, B \in \f \}$ and consider the convergence modes in terms of the $\M$-outer measure. 
The convergence modes with respect to the $\M$-outer measure reduce to the usual modes, respectively,
whenever measurability is available.

However,
even with the $\M$-outer measure,
the first conclusion of Theorem 1  does not necessarily hold in the presence of a measurability issue.
Indeed,
a consideration over rational translations of a usual Vitali set $V$ (which is not Lebesgue measurable) in $[0,1]$, 
whose elements are the components of a tuple of the Cartesian product $\bigtimes_{A \in [0, 1]/R}A$ where the product extends over all the  elements of the quotient space 
$[0,1]/R$ with respect to the equivalence relation $R \   \subset [0,1]^{2}$ defined by declaring that $xRy$ iff $x-y \in \Q$,
with full outer measure would lead to a counterexample. 
Here we certainly acknowledge the axiom of choice.
Although the counterexample thus obtained is somewhat of a routine nature,
for clarity we still elaborate and highlight a possible construction:

\begin{prop}
There are some Borel finite measure space and some sequence of  nonmeasurable functions from the measure space to $\R$
that converges pointwise but not in the outer measure.  
\end{prop}

\begin{proof}
Consider the unit interval $[0, 1]$ equipped with Lebesgue measure $\L$ restricted to the Borel subsets of $[0,1]$. 

It is (well-)known that there is some Vitali set $V \subset [0,1]$ such that $\L^{*}(V) = 1$.
Let $\{ q_{n} \}_{n \in \N} \ceq \Q \ \cap \ ]0,1[$;
for each $n \in \N$,
let $V_{n}$
be obtained from a rational translation $V + q_{n}$ of $V$
so that $\{ V_{n} \}$ is a partition of $[0,1]$.
Then each $V_{n}$ is  not Lebesgue measurable,
and each $V_{n}$ has Lebesgue-outer measure $\L^{*}(V_{n}) = 1$.

Let $f_{n} \ceq \I_{V_{n}}$  on $[0, 1]$
for each $n \in \N$;
then each $f_{n}$ is not Lebesgue measurable,
and $f_{n} \to 0$ pointwise. 
In particular,
(if informative) we have $f_{n} \to 0$ 
almost everywhere with respect to both $\L$ and $\L^{*}$.

However, 
we have $\L^{*}(|f_{n}| > \eps) = \L^{*}(V_{n}) = 1$ for all $n \in \N$ and all $0 < \eps < 1$;
so the sequence $(f_{n})$ does not converge in $\L^{*}$ to the zero function.
\end{proof}

Since the measure space $[0,1]$ considered in the above proof can be viewed as the probability space describing the uniform distribution concentrated on $[0, 1]$,
Proposition 1 has a probabilistic interpretation and hence is not terribly  artificial.

Proposition 1 prevents one from quickly generalizing Theorem 1,
which assumes the absence of a measurability issue,
to cover the case where a measurability issue is of concern.

It would be interesting to ask to what extent Theorem 1 still persists under indefinite measurability.

\end{document}